\documentclass[reqno]{amsart}
\usepackage{amssymb,amsmath,hyperref}
\usepackage{amsrefs, float}
\usepackage[foot]{amsaddr}
\usepackage{bbold,stackrel, multicol}
\usepackage{mathrsfs}

\newcommand{\Z}{{\textsf{\textup{Z}}}}

\newtheorem{thm}{Theorem}

\newtheorem{cor}[thm]{Corollary}
\newtheorem{defi}[thm]{Definition}
\newtheorem{rem}[thm]{Remark}
\newtheorem{nota}[thm]{Notation}

\newtheorem{princ}[thm]{Principle}

\newtheorem{ack}[thm]{Acknowledgement}

\newtheorem*{tempo*}{Template}

\newtheorem{lemma}[thm]{Lemma}

\newtheorem{claim}[thm]{Claim}

\newcommand\be{\begin{equation}}
\newcommand\ee{\end{equation}} 

\usepackage{amsmath,amsfonts} 

\usepackage[applemac]{inputenc}

\def\bdefi{\begin{defi}}
\def\edefi{\end{defi}}
\def\bnota{\begin{nota}\rm}
\def\enota{\end{nota}}

\def\ZF{\textup{\textsf{ZF}}}

\def\L{\textsf{\textup{L}}}

 \def\r{\mathbb{r}}

\def\RCA{\textup{\textsf{RCA}}}
\def\({\textup{(}}
\def\){\textup{)}}

\def\s{\textup{\textsf{s}}}
\def\RCAo{\textup{\textsf{RCA}}_{0}^{\omega}}
\def\ACAo{\textup{\textsf{ACA}}_{0}^{\omega}}

\def\WKL{\textup{\textsf{WKL}}}

\def\WWKL{\textup{\textsf{WWKL}}}
\def\bye{\end{document}}
\def\N{{\mathbb  N}}
\def\Q{{\mathbb  Q}}
\def\R{{\mathbb  R}}

\def\SS{\textup{\textsf{S}}}

\def\di{\rightarrow}

\def\asa{\leftrightarrow}

\def\ACA{\textup{\textsf{ACA}}}

\def\QFAC{\textup{\textsf{QF-AC}}}
\def\AC{\textup{\textsf{AC}}}

\def\CLC{\textup{\textsf{ClC}}}
\def\OC{\textup{\textsf{OC}}}

\def\NCC{\textup{\textsf{NCC}}}

\def\NCC{\textup{\textsf{NCC}}}

\def\HBC{\textup{\textsf{HBC}}}

\def\seq{\textup{\textsf{seq}}}

\usepackage{graphicx}
\usepackage{tikz}
\usetikzlibrary{matrix}
\usepackage{comment,tikz-cd}

\numberwithin{equation}{section}
\numberwithin{thm}{section}

\begin{document}
\title{On sequential theorems in Reverse Mathematics}
\author{Dag Normann}
\address{Department of Mathematics, The University of Oslo, Norway, P.O. Box 1053, Blindern N-0316 Oslo}
\email{dnormann@math.uio.no}
\author{Sam Sanders}
\address{Department of Philosophy II, RUB Bochum, Germany}
\email{sasander@me.com}
\keywords{Reverse Mathematics, higher-order arithmetic, Heine-Borel theorem, semi-continuity, sequential theorems.}
\subjclass[2010]{03B30, 03F35}
\begin{abstract}
Many theorems of mathematics have the form that for a certain \emph{problem}, e.g.\ a differential equation or polynomial (in)equality, there exists a \emph{solution}.
The \emph{sequential} version then states that for a \emph{sequence} of problems, there is a \emph{sequence} of solutions.
The original and sequential theorem can often be proved via the same (or similar) proof and often have the same (or similar) logical properties, esp.\ if everything is formulated in the language of second-order arithmetic.  In this paper, we identify basic theorems of third-order arithmetic, e.g.\ concerning semi-continuous functions, such that the sequential versions have very different logical properties.  
In particular, depending on the constructive status of the original theorem, very different and independent choice principles are needed. 
Despite these differences, the associated Reverse Mathematics, working in Kohlenbach's higher-order framework, is rather elegant and is still based at the core on \emph{weak K\"onig's lemma}.  
\end{abstract}


\maketitle
\thispagestyle{empty}

\section{Introduction and preliminares}\label{intro}
In a nutshell, we show that theorems and their \emph{sequential versions} can be behave rather differently in Kohlenbach's higher-order Reverse Mathematics (RM for short), in contrast to second-order RM.  Nonetheless, the associated third-order RM is quite elegant and based on \emph{weak K\"onig's lemma} at its core.  We assume familiarity with Kohlenbach's higher-order RM (\cite{kohlenbach2}), including the base theory $\RCAo$.  

\smallskip
\noindent
In more detail, a theorem $T$ of mathematics often has the syntactical form: 
\begin{center}
\emph{for all $x$ satisfying $P(x)$, there exists $y$ satisfying $Q(x, y)$}.  
\end{center}
The \emph{sequential} version of $T$, denoted $T^{\seq}$, then is formulated as follows
\begin{center}
\emph{for a \emph{sequence} $(x_{n})_{n\in \N}$ such that $(\forall n\in \N)P(x_{n})$, there is a \emph{sequence} $(y_{m})_{m\in \N}$ such that $(\forall m\in \N)Q(x_{m}, y_{m})$.}
\end{center}
Kohlenbach shows in \cite{kooltje} that \emph{weak K\"onig's lemma}, denoted $\WKL_{0}$, is equivalent to $\WKL_{0}^{\seq}$ over (what we now call) the base theory $\RCAo$ from \cite{kohlenbach2}.
Other references were sequential theorems are studied in RM are \cites{fuji1,fuji2,hirstseq,dork2,dork3, kooltje, simpson2,damurm,yokoyamaphd, polahirst}.

\smallskip

In general, the theorems $T$ and $T^{\seq}$ can often be proved via the same (or similar) proof and often have the same (or similar) logical properties, esp.\ if the former are formulated in the language of second-order arithmetic (see Remark \ref{LEM2}).  In this paper, we identify basic theorems $T$ of third-order arithmetic, e.g.\ concerning semi-continuous functions, 
such that the sequential versions $T^{\seq}$ have very different logical properties.  
In particular, depending on the constructive status of the original theorem, very different and rather independent choice principles are needed.  Representative examples are the Heine-Borel theorem and its contrapositive, the Cantor intersection theorem, where the latter requires a fragment of quantifier-free countable choice, called $\CLC$ below, and the former a fragment of numerical choice involving a universal real quantifier, called $\OC^{0,0}$ below.  
Despite these differences, the associated RM of $T$ and $T^{\seq}$, in Kohlenbach's framework \cite{kohlenbach2}, is rather elegant \emph{and} is still based at the core on \emph{weak K\"onig's lemma}, the second Big Five system of RM.  
As a side-result, we obtain new equivalences over $\RCAo$, as opposed to previous equivalences over extensions of the latter with countable choice\footnote{Many equivalences provable over $\RCAo+\QFAC^{0,1}$, do not go through over $\RCAo$ (\cite{dagsamV}); here, $\QFAC^{0,1}$ is $(\forall Y^{2})\big[(\forall n\in \N)(\exists f\in \N^{\N})(Y(f, n)=0)\di(\exists \Phi^{0\di 1})(\forall n\in \N)(Y(\Phi(n), n)=0)\big]$.\label{lola}}  (see e.g.\ \cite{dagsamXIV}), as well as a connection to \emph{hyperarithmetical analysis} by Remark \ref{hecho}.

\smallskip

Finaly, the RM-study of semi-continuous functions is long overdue as the latter are central to various sub-fields of analysis, including PDEs, as discussed in detail in \cite{roytype}.
As shown in \cite{dagsamXIV, samrep}, the coding of usco functions as in \cite{ekelhaft, pekelhaft} dramatically changes the logical strength of basic properties of usco functions. 
The results in this paper shall be seen to provide more evidence for this observation, based on the independence results for $\CLC$ and $\OC^{0,0}$, as discussed in Remark \ref{indep}.

\section{Main results}\label{main}
\subsection{Introduction}\label{xintro}
In this section, we prove our main results as follows.   We assume basic familiarity with RM, esp.\ Kohlenbach's approach from \cite{kohlenbach2}.
\begin{itemize}
\item In Section \ref{PRUL}, we introduce some basic definitions that cannot be found in Kohlenbach's founding paper \cite{kohlenbach2} of higher-order RM.
\item In Section \ref{TP}, we obtain some equivalences involving $\WKL_{0}$ and basic properties of usco functions.  
\item In Section \ref{TP2}, we obtain some equivalences involving $\WKL_{0}$ and sequential versions of the theorems studied in Section \ref{TP}.
\item In Section \ref{varvka}, we discuss some variations of the aforementioned results.    
\end{itemize}
The equivalences in Section \ref{TP2} for sequential theorems split into two categories.
\begin{itemize}
\item The RM of the sequential version of the \emph{Heine-Borel theorem} involves a non-trivial instance of \emph{numerical choice}, called $\OC^{0,0}$.
\item The RM of the sequential version of the \emph{Cantor intersection theorem} involves a non-trivial instance of \emph{countable choice}, called $\CLC$. 
\end{itemize}
The same observation holds for principles with the same syntactical form, where we note that the Heine-Borel theorem is the (classical) contraposition of the Cantor intersection theorem.    
We recall that $\WKL_{0}$ (resp.\ the Cantor intersection theorem) are \emph{rejected} in constructive mathematics while the \emph{contraposition} of $\WKL_{0}$, called the (weak/decidable) \emph{fan theorem} (resp.\ the Heine-Borel theorem), is \emph{semi-constructive}, as it is accepted in Brouwerian intuitionistic mathematics (\cite{brich, ishi1}).  

\smallskip

In conclusion, the behaviour of sequential theorems depends on the constructive status of the original theorem.  
Similar observations are made in \cite{sayo, kohlenbach2, dork1,koco} and Remark \ref{LEM2}, but the behaviour in this paper can be said to be `more pronounced/wild' as $\CLC$ and $\OC^{0,0}$ are independent over fairly strong systems and \emph{hard\footnote{The systems $\Z_{2}^{\omega}$ and $\Z_{2}^{\Omega}$ from Section \ref{PRUL} are both conservative extensions of $\Z_{2}$.  However, $\Z_{2}^{\omega}$ cannot prove $\CLC$ and $\Z_{2}^{\omega}+\QFAC^{0,1}$ cannot prove $\OC^{0,0}$, while $\Z_{2}^{\Omega}$ does prove both.} to prove} by Theorems \ref{hencka} and \ref{henckb}. 

\subsection{Preliminaries}\label{PRUL}
We introduce some basic notions, like the definition of open set, that cannot be found in the founding text of higher-order RM \cite{kohlenbach2}.
Definitions take place in $\RCAo$ unless explicitly stated otherwise. 

\smallskip

First of all, open sets are represented in second-order RM by \emph{countable unions of basic open balls}, namely as in \cite{simpson2}*{II.5.6}.  
In light of \cite{simpson2}*{II.7.1}, \emph{\(codes for\) continuous functions} provide an equivalent representation (over $\RCA_{0}$). 
In particular, the latter second-order representation is exactly the following definition restricted to (codes for) continuous functions (\cite{simpson2}*{II.6.1}).  
\bdefi\label{bchar}
An \emph{open set} $U\subset \R$ is given by $h_{U}:\R\di \R$ where we say `$x\in U$' if and only if $h_{U}(x)>0$ for any $x\in \R$ and where $y\in U$ implies $(\exists N\in \N)(\forall z\in B(y, \frac{1}{2^{N}})(z\in U )$.  
A set is closed if the complement is open. 
\edefi
Since codes for continuous functions denote third-order functions in $\RCAo$ (see \cite{dagsamXIV}*{\S2}), Def.\ \ref{bchar} includes the second-order definition.  To be absolutely clear, combining \cite{dagsamXIV}*{Theorem 2.2} and \cite{simpson2}*{II.7.1}, $\RCAo$ immediately proves
\begin{center}
\emph{a code $U$ for an open set represents an open set in the sense of Def.\ \ref{bchar}}. 
\end{center}
Assuming Kleene's quantifier $(\exists^{2})$ from the next paragraph, Def.\ \ref{bchar} is equivalent to the existence of a characteristic function for open sets; the latter definition is used in e.g.\ \cite{dagsamVII, dagsamXV}.  
Thus, we may take the representation function $h_{U}$ to be \emph{lower semi-continuous} (see Def.\ \ref{flung} below) in Def.\ \ref{bchar}, even\footnote{Since $\RCAo$ is a classical system, we may invoke the law of excluded middle as in $(\exists^{2}\vee \neg(\exists^{2})$.  In case $(\exists^{2})$, note that the characteristic function of $U$ is lower semi-continuous.  In case $\neg(\exists^{2})$, all functions on the reals are continuous by \cite{kohlenbach2}*{Prop.\ 3.12}, and hence $h_{U}$ is (lower semi-) continuous.} in $\RCAo$.  

\smallskip

Secondly, full second-order arithmetic $\Z_{2}$ is the `upper limit' of second-order RM.  The systems $\Z_{2}^{\omega}$ and $\Z_{2}^{\Omega}$ are conservative extensions of $\Z_{2}$ by \cite{hunterphd}*{Cor.\ 2.6}. 
The system $\Z_{2}^{\Omega}$ is $\RCAo$ plus Kleene's quantifier $(\exists^{3})$ (see e.g.\ \cite{dagsamXIV} or \cite{hunterphd}), while $\Z_{2}^{\omega}$ is $\RCAo$ plus $(\SS_{k}^{2})$ for every $k\geq 1$; the latter axiom states the existence of a functional $\SS_{k}^{2}$ deciding $\Pi_{k}^{1}$-formulas in Kleene normal form.  We write $\ACAo$ for $\RCAo+(\exists^{2})$ where the latter is as follows 
\be\label{muk}\tag{$\exists^{2}$}
(\exists E:\N^{\N}\di \{0,1\})(\forall f \in \N^{\N})\big[(\exists n\in \N)(f(n)=0) \asa E(f)=0    \big]. 
\ee
Over $\RCAo$, $(\exists^{2})$ is equivalent to the existence of Feferman's $\mu$ (see \cite{kohlenbach2}*{Prop.\ 3.9}), defined as follows for all $f\in \N^{\N}$:
\[
\mu(f):= 
\begin{cases}
n & \textup{if $n$ is the least natural number such that $f(n)=0$, }\\
0 & \textup{ if $f(n)>0$ for all $n\in \N$}
\end{cases}
.
\]
Thirdly, we shall study Baire's notion of semi-continuity first introduced in \cite{beren}.
\bdefi\label{flung} 
For $f:[0,1]\di \R$, we have the following definitions:
\begin{itemize}
\item $f$ is \emph{upper semi-continuous} at $x_{0}\in [0,1]$ if for any $k\in \N$, there is $N\in \N$ such that $(\forall y\in B(x_{0}, \frac{1}{2^{N}}))( f(y)< f(x_{0})+\frac{1}{2^{k}} )$,
\item $f$ is \emph{lower semi-continuous} at $x_{0}\in [0,1]$ if for any $k\in \N$, there is $N\in \N$ such that $(\forall y\in B(x_{0}, \frac{1}{2^{N}}))( f(y)> f(x_{0})-\frac{1}{2^{k}} )$,
\item $f$ is \emph{Baire $1$} if it is the pointwise limit of a sequence of continuous functions.
\end{itemize}
\edefi
Regarding the third item, the sequence of continuous functions is called the `Baire 1 representation of $f$'.  
We use the common abbreviations `usco' and `lsco' for the previous notions.  We say that `$f:[0,1]\di \R$ is usco' if $f$ is usco at every $x\in [0,1]$.  A set $C\subset [0,1]$ is closed (resp.\ open) if and only if the characteristic function $\mathbb{1}_{C}$ is usco (resp.\ lsco).  Since this equivalence goes through in weak systems, properties of usco functions are often equivalent to properties of closed sets, and vice versa. 

\smallskip

Finally, we discuss the behaviour of (second-order) sequential theorems in detail in the following remark, which was pointed out to us by Ulrich Kohlenbach.
\begin{rem}[Sequential theorems and the law of excluded middle]\label{LEM2}\rm
As noted in Section \ref{intro}, $\WKL_{0}$ is equivalent to $\WKL_{0}^{\seq}$ over $\RCA_{0}$ (\cite{kooltje}).   
By contrast, the sequential form $T^\seq$ of a theorem $T$ (usually) is stronger than $T$ whenever 
the proof of $T$ needs some instance $A$ of the law of excluded middle (\textsf{LEM}) for which the sequential form $A^{\seq}$ is stronger than what is needed to prove $T$.  Well-known example are Ramsey's theorem for pairs and {weak weak K\"onig's lemma} (see Section~\ref{zol}), while a more recent example may be found in \cite{koco}.  
Indeed, in the latter, the regularity of continuous mappings on compact spaces is established in $\WKL_0$ while the existence of a \emph{modulus} of regularity is seen to require $\ACA_0$.
The first proof in $\WKL_{0}$ makes use of $\Sigma^0_1$-\textsf{LEM}, for which the sequential form is $\Sigma_{1}^{0}$-comprehension, and hence $\ACA_{0}$.
\end{rem}

\subsection{Some equivalences involving weak K\"onig's lemma}\label{TP}
In this section, we obtain some new equivalences that are part of the RM of $\WKL_{0}$, including basic properties of semi-continuous functions.  
The results in \cite{martino} suggest that semi-continuity is the largest class that can be used here. 
The sequential versions of the associated principles shall be studied in Section \ref{TP2}.   

\smallskip

First of all, various versions of the countable Heine-Borel theorem are equivalent to $\WKL_{0}$ by \cite{brownphd}*{Lemma 3.13} or \cite{simpson2}*{IV.1.6}. 
We have studied these principles \emph{for open/closed sets without codes} in \cite{dagsamVII, samhabil}, including the following.
\begin{princ}[$\HBC_{\s}$]
Let $C\subseteq [0,1]$ be closed and let $(O_{n})_{n\in \N}$ be a sequence of open sets with $C\subseteq \cup_{n\in \N}O_{n}$.  Then $C\subseteq \cup_{n\leq n_{0}}O_{n}$ for some $n_{0}\in \N$.
\end{princ}
\noindent
We let $\HBC$ be $\HBC_{\s}$ restricted to sequences of \emph{basic open intervals}.  

\smallskip

Secondly, the following theorem suggests that the RM of $\HBC_{\s}$ is \emph{close} to that of $\WKL_{0}$, but not over $\RCAo$ (or the much stronger $\Z_{2}^{\omega}$). 
We believe that $\HBC$ does not imply $\HBC_{\s}$ over $\Z_{2}^{\omega}$.  We note that item \eqref{itemaz} is a generalisation of the `positivity' theorem from constructive reverse mathematics (\cite{brich}*{Cor.\ 2.8}).
\begin{thm}\label{hench}~
Over $\RCAo$, the following are equivalent:
\begin{enumerate}
\renewcommand{\theenumi}{\alph{enumi}}
\item a usco function $f:[0,1]\di \R$ is bounded above,\label{itema}
\item for any lsco $f:[0,1]\di \R^{+}$, we have $(\exists N\in \N)(\forall x\in [0,1])(f(x)>\frac{1}{2^{N}})$,\label{itemaz}
\item \(Heine-Borel\) for a sequence $(O_{n})_{n\in \N}$ of open sets such that $\cup_{n\in \N}O_{n}$ covers $[0,1]$, there is $n_{0}\in \N$ such that $\cup_{n\leq n_{0}}O_{n}$ covers $[0,1]$, \label{itemb}
\item \(Cantor intersection theorem\) for a sequence $(C_{n})_{n\in \N}$ of non-empty closed sets with $C_{n+1}\subseteq C_{n}\subseteq [0,1]$ for all $n\in \N$, $\cap_{n\in \N}C_{n}\ne \emptyset$,\label{itembp}
\item \(pointwise and uniform domination\) let $(f_{n})_{n\in \N}$ be an increasing sequence of lsco functions.  Then for usco $g:[0,1]\di \R$ and $I\equiv [0,1]$, we have: \label{fargo}
\be\label{dreuk}
(\forall x\in I)(\exists n\in \N)(f_{n}(x)>g(x))\di (\exists m\in \N)(\forall x\in I)(f_{m}(x)>g(x)),
\ee
\item the principle $\HBC_{\s}$,\label{itemc}
\item for usco $f:[0,1]\di \R$ with supremum $y$, there is $x\in [0,1]$ with $f(x)=y$,\label{itemd}
\item for usco $f:[0,1]\di \R$ with supremum $y$ and \textbf{at most one maximum}\footnote{We say that $f:[0,1]\di \R$ with supremum $y$ has \emph{at most one maximum} in case $(\forall x, x'\in [0,1])( x\ne x'\di f(x)<y\vee f(x')<y) $, a notion from constructive analysis (see e.g.\ \cite{ishberg}).}, there is $x\in [0,1]$ with $f(x)=y$.\label{itemdz}
\end{enumerate}
Over $\RCAo+\QFAC^{0,1}$, items \eqref{itema}-\eqref{itemdz} are equivalent to $\WKL_{0}$ and to 
\begin{enumerate}  
\renewcommand{\theenumi}{\alph{enumi}}
\setcounter{enumi}{8}
\item the principle $\HBC$, \label{itemf}
\item for usco $f:[0,1]\di \R$ with a Baire 1 representation, there is $x\in [0,1]$ with $(\forall y\in [0,1])(f(y)\leq f(x))$,\label{itemddd}

\item for usco $f:[0,1]\di \R$ with essential\footnote{A real $y\in \R$ is the \emph{essential supremum} of $f:[0,1]\di \R$ in case $\{ x\in [0,1]: f(x)\geq y\}$ has measure zero and $\{ x\in [0,1]: f(x)\geq y-\frac{1}{2^{k}}\}$ has positive measure for all $k\in \N$.  Notions like `measure zero' can be expressed in $\RCAo$ and $\RCA_{0}$ without recourse to the Lebesgue measure.
} supremum $y$, there is $x\in [0,1]$ with $f(x)=y$.\label{iteme}
\end{enumerate}
The system $\Z_{2}^{\omega}$ cannot prove items \eqref{itema}-\eqref{itemf} while $\Z_{2}^{\Omega}$ proves items \eqref{itema}-\eqref{iteme}. 
\end{thm}
\begin{proof}
First of all, the equivalence between items \eqref{itemb} and \eqref{itembp} (resp.\ items \eqref{itema} and \eqref{itemaz}) amounts to a manipulation of definitions.   
Now assume item \eqref{itema} and let $(O_{n})_{n\in \N}$ be a sequence of open sets such that $\cup_{n\in \N}O_{n}$ covers $[0,1]$.  
Apply\footnote{Technically, we apply $\QFAC^{1,0}$ to the formula at hand where `$(\forall x\in [0,1])$' is replaced by `$(\forall f\in 2^{\N})$' and where `$x$' is replaced by `$\r(f)$', which is $\sum_{n=0}^{\infty}\frac{f(i)}{2^{i+1}}$ by definition.\label{floef}} $\QFAC^{1,0}$, included in $\RCAo$, to the following formula 
\be\label{kion}
(\forall x\in [0,1])(\exists n\in \N)(x\in O_{n})
\ee
and let $f:[0,1]\di \R$ be the associated function.  
By definition, $f$ is usco and therefore bounded above, i.e.\ item \eqref{itemb} follows.  
Now assume item \eqref{itemb} and let $f:[0,1]\di \R$ be usco. 
Essentially by definition, the set $C_{n}:=\{  x\in [0,1]: f(x)\geq n\}$ is closed.   Note that $O_{n}:=[0,1]\setminus C_{n}$ is such that $\cup_{n}O_{n}$ covers $[0,1]$.
Applying item~\eqref{itemb}, we find an upper bound to $f$, i.e.\ item \eqref{itema} follows.  

\smallskip

Note that in item \eqref{fargo}, we may assume $g(x)=0$ for all $x\in [0,1]$ as $h_{n}(x)=f_{n}(x)-g(x)$ is also lsco.  
To prove item \eqref{fargo} from item \eqref{itemb}, note that $O_{n}:=\{x\in [0,1]:f_{n}(x)>0   \}  $ is open for lsco $f_{n}$.
Moreover, the antecedent of \eqref{dreuk} implies that $\cup_{n\in \N}O_{n}$ covers $[0,1]$.  Item \eqref{fb} provides $n_{0}\in \N$ such that $\cup_{n\leq n_{0}}O_{n}$ covers $[0,1]$. 
Since $(f_{n})_{n\in \N}$ is increasing, $m=n_{0}$ satisfies the consequent of \eqref{dreuk}.  For the reversal, let $(O_{n})_{n\in \N}$ be an open covering of $[0,1]$ and let $f_{n}:[0,1]\di \R$ be the (lsco) representation of the open set $\cup_{i\leq n}O_{k}$.   Then $(f_{n})_{n\in \N}$ is increasing and satisfies $(\forall x\in [0,1])(\exists n\in \N)(f_{n}(x)>0)$.  
By item \eqref{fargo}, there is $m_{0}\in \N$ with $(\forall x\in [0,1])(f_{m_{0}}(x)>0)$, implying that $[0,1]\subset \cup_{m\leq m_{0}} O_{m}$.

\smallskip

Clearly, item \eqref{itemc} implies item \eqref{itemb} and we now show that item \eqref{itema} implies item \eqref{itemc}.  
To this end, let $C$ be closed and let $(O_{n})_{n\in \N}$ be an open covering of $[0,1]$.  
In case all functions $h_{C}, h_{O_{n}}$ from Def.\ \ref{bchar} are continuous, they have RM-codes by \cite{dagsamXIV}*{Cor.\ 2.5}.
In this case, item \eqref{itemc} reduces to a second-order statement, which follows from \eqref{itema} by \cite{brownphd}*{Lemma 3.13} and \cite{dagsamXIV}*{Theorem 2.8}.  
In case one of these functions is discontinuous, we obtain $(\exists^{2})$ by \cite{kohlenbach2}*{Prop.\ 3.12}.
Now use (the equivalent) Feferman's $\mu$ to define $f:[0,1]\di \R$ as follows
\be\label{edo}
f(x):=
\begin{cases}
0 & x\not \in C\\
n & \textup{$n$ is the least natural number such that $x\in O_{n}$}
\end{cases}.
\ee
Note that the axiom of (function) extensionality as in $x=_{\R}y\di f(x)=_{\R}f(y)$ holds by definition.  
We now show that $f$ is usco, i.e.\ we can apply item \eqref{itema} to obtain $\HBC_{\s}$.  Since $C$ is closed, $f(x)=0$ implies that $f(y)=0$ for all $y\in B(x, \frac{1}{2^{N}})$ and for some $N\in \N$.
In particular, $f$ is continuous at $x$ if $x\not\in C$.  Now, in case $x\in C$ and $f(x)=n$, we have $f(y)\leq n $ for $y\in O_{n}$ by definition, i.e.\ $f$ is usco at $x$.  

\smallskip

Now assume item \eqref{itembp} and let $f:[0,1]\di \R$ be usco with supremum $y$.  
By definition $(\forall n\in \N)(\exists x\in [0,1])(f(x)\geq y-\frac{1}{2^{n}} )$ and the following sequence    
\[\textstyle
E_{n}:= \{x\in [0,1]: f(x)\geq y-\frac{1}{2^{n}} \}
\]
consists of closed and non-empty sets that are nested.  Apply item \eqref{itembp} and let $z \in \cap_{n\in \N}E_{n}$.   
Since $z\in E_{n}$ for all $n\in \N$, we must have $f(z)=y$, i.e.\ item \eqref{itemd} follows.  
Now assume item \eqref{itemd} and suppose item \eqref{itema} is false, i.e.\ there is a usco $f:[0,1]\di \R$ that is unbounded above.
Note that $f$ is necessarily discontinuous, i.e.\ $(\exists^{2})$ follows by \cite{kohlenbach2}*{Prop.\ 3.12}.
Then use $(\exists^{2})$ to define usco $g(x):= \sum_{n=0}^{\lfloor f(x)\rfloor}\frac{1}{n!}$ which satisfies $\sup_{x\in [0,1]}g(x)=e$ and $g(y)<e$ for all $y\in [0,1]$.  
This contradicts our assumption (of item \eqref{itemd}) and item \eqref{itembp} follows.  Note that the previous proof also goes through for item \eqref{itemdz} as $g$ (trivially) has at most one maximum. 
The equivalences over $\RCAo$ are now finished. 

\smallskip

For the equivalences over $\RCAo+\QFAC^{0,1}$, we now derive item \eqref{itema} in $\RCAo+\WKL_{0}+\QFAC^{0,1}$.
If $f:[0,1]\di \R$ is unbounded and usco, use $\QFAC^{0,1}$ to obtain a sequence $(x_{n})_{n\in \N}$ such that $f(x_{n})>n$ for all $n\in \N$.  
Since $(\exists^{2})\di \ACA_{0}$, sequential compactness (\cite{simpson2}*{III.2.2}) provides a convergent sub-sequence, say with limit $y\in[0,1]$.  Clearly, $f$ cannot be usco at $y$, 
a contradiction, and the former must be bounded, i.e.\ item \eqref{itema} follows. 
By \cite{dagsamXIV}*{Theorem 2.9}, a bounded Baire 1 function has a supremum in $\RCAo+\WKL$, i.e.\ item \eqref{itemddd} follows from item \eqref{itema}.  

\smallskip

For item \eqref{iteme}, one verifies that $\tilde{f}:[0,1]\di \R$ is usco if $f:[0,1]\di \R$ is usco:
\[
\tilde{f}(x):=
\begin{cases}
f(x)  & f(x)\leq y \\
y & \textup{otherwise}
\end{cases}.
\]
Hence, if $f$ has essential supremum $y$, then $\tilde{f}$ has supremum $y$.  
Thus, item \eqref{itemd} implies item \eqref{iteme}, and the latter immediately 
implies $\WKL_{0}$ via the special case for continuous functions.  

\smallskip

For the negative result, $\Z_{2}^{\omega}$ cannot prove $\HBC$ by \cite{dagsamVII}*{Theorem 3.5}.
The final sentence follows from \cite{dagsamVII}*{Theorem 4.5}, where the latter establishes that open sets have (second-order) codes in $\Z_{2}^{\Omega}$.  
\end{proof}
We recall that $\Z_{2}^{\omega}$ cannot prove the general existence of the supremum of usco functions by \cite{dagsamXIV}*{\S2.8.1}, explaining the absence of this statement in Theorem~\ref{hench}.

\smallskip

In our opinion, the equivalences between items \eqref{itema}-\eqref{itemc} in Theorem \ref{hench} are rather elegant and the only `blemish' is the need for a stronger base theory than $\RCAo$ to obtain an equivalence to $\WKL_{0}$.  
Of course, we do not \emph{need} $\QFAC^{0,1}$ in Theorem \ref{hench}: the weaker axiom $\NCC$ from \cite{dagsamIX}, provable in $\Z_{2}^{\Omega}$, suffices (exercise!).
Nonetheless, the base theory $\RCAo+\NCC$ is still an highly non-trivial extension of $\RCAo$.  An elegant solution may be found in Section \ref{TP2}.    

\smallskip

Finally, we discuss some variations of the above results. 
Now, Theorem \ref{hench} shows that certain higher-order generalisations of the RM of $\WKL_{0}$ go (slightly) beyond the latter.  
This need not be the case: over $\RCAo$, $\WKL_{0}$ is equivalent to the higher-order generalisation of $\Sigma_{1}^{0}$-separation (\cite{simpson2}*{I.11.7}), where $\varphi_{i}(n)$ is replaced by $(\exists x\in [0,1])(x\in O_{n}^{i})$ and where the latter sets are open.  Moreover, $\WKL_{0}$ is closely related to $\WWKL_{0}$ where the latter is the former restricted to trees of positive measure; we briefly sketch variations of the above results for the system $\WWKL_{0}$ in Section \ref{varvka}.  The latter also discusses the role of \emph{Cousin's lemma}, a version of the Heine-Borel theorem.  

\subsection{Sequential theorems and weak K\"onig's lemma}\label{TP2}
We obtain some equivalences involving $\WKL_{0}$ and sequential versions of the theorems studied in Section~\ref{TP}.  As noted in Section \ref{xintro}, the behaviour of sequential theorems depends on the constructive status of the original theorems.  In particular, the Heine-Borel theorem and Cantor intersection theorem are classically equivalent over $\RCAo$, but the sequential versions require different choice principles by Theorems \ref{hencka} and \ref{henckb}.  

\subsubsection{Sequential Cantor intersection theorem}\label{TPSS}
In this section, we study the sequential Cantor intersection theorem.  As it turns out, the latter has a rather elegant connection to hyperarithmetical analysis by Remark \ref{hecho}.  

\smallskip

First of all, the following principle is essential.
\begin{princ}[$\CLC$]
Let $(C_{n})_{n\in \N}$ be a sequence of non-empty closed sets in $[0,1]$.  There is a sequence $(x_{n})_{n\in \N}$ such that $x_{n}\in C_{n}$ for all $n\in \N$.
\end{princ}
Now, $\CLC$ follows from the \emph{Lindel\"of lemma} in its original form for closed subsets of $\R$ (\cite{blindeloef}).
By \cite{simpson2}*{IV.1.8}, $\WKL_{0}$ proves $\CLC$ \emph{restricted to codes for closed sets}.  
By contrast, $\CLC$ is rather hard to prove by Theorem \ref{hencka} and Remark \ref{hecho}.  
\begin{thm}\label{hencka}~
\begin{itemize}
\item The system $\Z_{2}^{\omega}$ cannot prove $\CLC$.
\item The system $\Z_{2}^{\Omega}$ or $\RCAo+\QFAC^{0,1}$ proves $\CLC$.
\item The equivalence between item \eqref{itema} of Theorem \ref{hench} and $\WKL_{0}$ is provable over $\RCAo+\CLC$. 
\end{itemize}
\end{thm}
\begin{proof}
The second item follows via the usual interval-halving technique.  
The first item follows from the third item; indeed, if $\Z_{2}^{\omega}$ proves $\CLC$, then the third item implies that $\Z_{2}^{\omega}$ also proves item \eqref{itema} of Theorem \ref{hench}, which contradicts the final sentence of Theorem \ref{hench}.  
To establish the third item, we now derive item \eqref{itema} from Theorem \ref{hench} in $\RCAo+\WKL+\CLC$.
Suppose $f:[0,1]\di \R$ is unbounded and usco.  By \cite{dagsamXIV}*{Theorem 2.8}, continuous functions on the unit interval are bounded, i.e.\ $f$ must be discontinuous, yielding $(\exists^{2})$ by \cite{kohlenbach2}*{Prop.\ 3.12}.
Use $(\exists^{2})$ to define the following sequence of closed sets:
\be\label{zank}
 E_{n}:=\{x\in [0,1]: f(x)\geq n\}.
\ee
That $E_{n}$ is closed follows immediately from the fact that $f$ is usco; that $E_{n}$ is non-empty follows by assumption on $f$.  
Now apply $\CLC$ to obtain a sequence $(x_{n})_{n\in \N}$ such that $f(x_{n})\geq n$ for all $n\in \N$.  
Since $(\exists^{2})\di \ACA_{0}$, sequential compactness (\cite{simpson2}*{III.2.2}) provides a convergent sub-sequence, say with limit $y\in[0,1]$.  Clearly, $f$ cannot be usco at $y$, 
a contradiction, and the former must be bounded.  Note that \eqref{zank} forms a decreasing sequence to finish the proof. 
\end{proof}
One could argue that $\RCAo+\CLC$ is an acceptable base theory as the coding of open sets renders $\CLC$ \emph{restricted to codes} provable in $\WKL_{0}$ by \cite{simpson2}*{IV.1.8}.  
Following Theorem \ref{hench}, $\RCAo+\CLC$ is a more elegant base theory than $\RCAo+\QFAC^{0,1}$, as the latter is not provable in $\ZF$ while $\CLC$ is, namely by Theorem \ref{hencka}.   

\smallskip

Secondly, while the previous considerations are important, the true purpose of $\CLC$ is revealed by Theorem \ref{henchmen2} where $\RCAo$ is again the base theory. 
Note that items \eqref{itf} and \eqref{itg} in Theorem \ref{henchmen2} are sequential versions of the maximum principle for usco functions, i.e.\ item \eqref{itemd} in Theorem \ref{hench}. 
\begin{thm}\label{henchmen2}
Over $\RCAo$, the following are equivalent.
\begin{enumerate}
\renewcommand{\theenumi}{\alph{enumi}}
\item The combination of $\WKL_{0}$ and $\CLC$.\label{ite}
\item  Let $f_{n}:([0,1]\times \N)\di \R$ be usco and with supremum $y\in \R$ for all $n\in \N$.  Then there is $(x_{n})_{n\in \N}$ such that $f_{n}(x_{n})=y$ for all $n\in \N$.\label{itf}
\item Let usco $f:\R\di \R$ and the sequence $(\sup_{x\in [n, n+1]}f(x))_{n\in \N}$ be given.  There is $(x_{n})_{n\in \N}$ with $x_{n}\in [n, n+1]\wedge f(x_{n})=\sup_{x\in [n, n+1]}f(x)$ for all $n\in \N$.\label{itg}
\item The previous item with fixed $y=\sup_{x\in [n, n+1]}f(x)$ for all $n\in \N$.
\item The principle $\CLC$ plus any of the items \eqref{itema}-\eqref{iteme} from Theorem \ref{hench}.\label{ith}
\item The sequential version of the Cantor intersection theorem. \label{iti}
\item Let $(C_{n})_{n\in \N}$ be a sequence of non-empty closed sets and let $f:[0,1]\di \R$ be continuous on $C_{n}$ with $y=\sup_{x\in C_{n}}f(x)$ for all $n\in \N$.  
Then there is $(x_{n})_{n\in \N}$ with $x_{n}\in C_{n} \wedge f(x_{n})=y$ for all $n\in \N$.\label{itj}
\end{enumerate}
\end{thm}
\begin{proof}
To obtain item \eqref{itf} from item \eqref{ite}, we first prove that for usco $f:[0,1]\di \R$ with supremum $y\in \R$, there is $x\in [0,1]$ with $f(x)=y$.
In case $f$ is continuous, this is immediate by \cite{dagsamXIV}*{Cor.\ 2.5} and the well-known second-order results.
In case $f$ is discontinuous, we obtain $(\exists^{2})$ by \cite{kohlenbach2}*{Prop.\ 3.12}.
Now, by definition, we have $(\forall n\in \N)(\exists x\in [0,1])(f(x)\geq y-\frac{1}{2^{n}} )$ and the following set    
\[\textstyle
E_{n}:= \{x\in [0,1]: f(x)\geq y-\frac{1}{2^{n}} \}
\]
is closed and non-empty.  Apply $\CLC$ to obtain a sequence $(x_{n})_{n\in \N}$ such that $(\forall n\in \N)(f(x_{n})\geq y-\frac{1}{2^{n}} )$.
Since $(\exists^{2})\di \ACA_{0}$, we have access to the second-order convergence theorems (see \cite{simpson2}*{III.2}).  
Let $(z_{n})_{n\in \N}$ be a convergent sub-sequence of $(x_{n})_{n\in \N}$, say with limit $y_{0}$.  Since $f$ is usco and $y$ its supremum, we have
\[\textstyle
y \ge f(y_0) \ge   \lim_{n \to \infty} f(z_{n}) \ge \lim_{n\di \infty} (y - \frac 1{2^{n}}) = y, 
\]
which implies $f(y_{0})=y$ as required.  Hence, for $(f_{n})_{n\in \N}$ a sequence of usco functions, the following set is non-empty and closed for all $n\in \N$:
\[
F_{n}:=\{x\in [0,1]:  f_{n}(x)\geq y\}
\]  
and $\CLC$ yields the sequence as in item \eqref{itf}; items \eqref{itg}-\eqref{ith} follow in the same way.   

\smallskip

To derive item \eqref{ite} from item \eqref{itf} (or items \eqref{itg}-\eqref{ith}), it suffices to obtain $\CLC$.  To this end, let $(C_{n})_{n\in \N}$ be a sequence 
of non-empty closed sets in $\R$.  Then $\mathbb{1}_{C_{n}}$ is a sequence of usco functions with supremum $1$ and applying item \eqref{itf} yields $\CLC$. 

\smallskip

Next, to prove the sequential version of the Cantor intersection theorem from item \eqref{ite}, let $(C_{n, m})_{n \in \N}$ be a sequence of non-empty closed sets such that $C_{n+1, m}\subset C_{n, m}\subset [0,1]$ for all $n,m \in \N$.  
If all functions representing $C_{n, m}$ are continuous, they have (a sequence of) codes assuming $\WKL$, by \cite{dagsamXIV}*{Cor.\ 2.5}.
The second-order proof using $\WKL$ and the Heine-Borel theorem now goes through.  
In case one of the functions representing $C_{n, m}$ is discontinuous, we obtain $(\exists^{2})$ by \cite{kohlenbach2}*{Prop.\ 3.12}.  
Then apply $\CLC$ to $(\forall n, m\in \N)(\exists x\in C_{n, m})$; the resulting sequence $(x_{n, m})_{n, m\in \N}$ has a sub-sequence for every $m\in \N$, by sequential compactness (\cite{simpson2}*{III.2}) as $(\exists^{2})\di \ACA_{0}$.    
In particular, there is non-decreasing $g\in \N^{\N}$ and $(y_{m})_{m\in \N}$ such that $(x_{g(n), m})_{n\in \N}$ is convergent to $y_{m}$, for all $m\in \N$.  Clearly, $y_{m}\in \cap_{n\in \N} C_{n, m}$ for all $m\in \N$, as required. 
That item \eqref{iti} implies $\CLC$ is immediate by considering a sequence of non-empty closed sets $(E_{k})_{k\in \N}$ and defining $C_{n, m}:= E_{m}$.  

\smallskip

Finally, note that item \eqref{itj} is a special case of item \eqref{ith}.  To derive $\CLC$ from the former, let $(C_{n})_{n\in \N}$ be a sequence of non-empty closed sets.  
Then $f(x):=1$ is continuous on $C_{n}$ with supremum equal to $1$.  
\end{proof}
Regarding the robustness of the equivalences in the previous theorem, observe that in item \eqref{ite} we can replace `$\WKL_{0}$' by the boundedness or supremum principle for most of the (many)
function classes studied in \cite{dagsamXIV}. 

\smallskip

Next, we show that $\CLC$ suffices to prove that closed sets are closed under limits.  
\begin{thm}[$\ACAo+\CLC$]\label{ploep}
The following are provable.
\begin{itemize}
\item  A set $C\subset [0,1]$ is closed if and only if it is sequentially\footnote{Any $C\subset [0,1]$ is sequentially closed if for any convergent sequence in $C$, the limit is in $C$.} closed. 
\item $\textsf{\textup{weak-}}\Sigma_{1}^{1}$-$\AC_{0}:$ for arithmetical $\varphi$, we have 
\[
(\forall n\in \N)(\exists! X\subset\N)\varphi(X, n)\di (\exists \Phi^{0\di 1})(\forall n\in \N)\varphi(\Phi(n), n).
\]
\end{itemize}
\end{thm}
\begin{proof}
For the first item, let $C\subset [0,1]$ be closed and let $(x_{n})_{n\in \N}$ be a sequence in $C$ converging to $y\in [0,1]$.  In case $y\not \in C$, there is $N\in \N$ such that $B(y, \frac{1}{2^{N}})\cap C=\emptyset$.  
This contradicts the fact that $(x_{n})_{n\in \N}$ converges to $y$, i.e.\ $C$ is also sequentially closed.  Now let $C\subset [0,1]$ be sequentially closed and suppose it is not closed, i.e.\ there is $y\not\in C$
such that $(\forall N\in \N)(\exists x\in C)(|x-y|<\frac{1}{2^{N}})$.  Apply $\CLC$ for $C_{n}:= [y-\frac{1}{2^{n}}, y+\frac{1}{2^{n}}]\cap C$ to obtain $(x_{n})_{n\in \N}$ in $C$ converging to $y$.  Since $C$ is sequentially closed, we have $y\in C$, a contradiction. 

\smallskip

For the second item, we may view $X\subset \N$ as elements of Cantor space and vice versa.
Using the well-known interval-halving method, $(\exists^{2})$ allows us to define a functional $\eta:[0,1]\di 2^{\N}$ 
such that $\eta(x)$ is the binary expansion of $x$, with a tail of zeros if relevant.  
Now use $(\exists^{2})$ to define the sequence of singletons $C_{n}:=\{x\in[0,1]:  \varphi(\eta(x), n)  \}$ where $\varphi$ is arithmetical.  
Applying $\CLC$, we obtain the sequence $\Phi$ as in $\textsf{\textup{weak-}}\Sigma_{1}^{1}$-$\AC_{0}$.
\end{proof}
Finally, we finish this section with a remark on \emph{hyperarithmetical analysis}.
\begin{rem}\label{hecho}\rm
The notion of \emph{hyperarithmetical set} (\cite{simpson2}*{VIII.3}) gives rise to the (second-order) definition of \emph{system/statement of hyperarithmetical analyis} (see e.g.\ \cite{monta2} for the exact definition), 
which includes systems like $\Sigma_{1}^{1}$-$\textsf{CA}_{0}$ (see \cite{simpson2}*{VII.6.1}).  Montalb\'an claims in \cite{monta2} that \textsf{INDEC}, a special case of \cite{juleke}*{IV.3.3}, is the first `mathematical' statement of hyperarithmetical analysis.  The latter theorem by Jullien can be found in \cite{aardbei}*{6.3.4.(3)} and \cite{roosje}*{Lemma 10.3}.  

\smallskip

The monographs \cites{roosje, aardbei, juleke} are all `rather logical' in nature and $\textsf{INDEC}$ is the \emph{restriction} of a higher-order statement to countable linear orders in the sense of RM (\cite{simpson2}*{V.1.1}), i.e.\ such orders are given by sequences.  
By the previous, $\ACAo+\CLC$ exists \emph{in the range of hyperarithmetical analysis}, namely sitting between $\RCAo+\textsf{weak}$-$\Sigma_{1}^{1}$-$\textsf{CA}_{0}$ and $\ACAo+\QFAC^{0,1}\equiv_{\L_{2}}\Sigma_{1}^{1}$-\textsf{CA}$_{0}$ by Theorem \ref{ploep}.  
Thus, $\ACAo$ plus items \eqref{ite}-\eqref{itj} from Theorem~\ref{henchmen2} are all (rather) natural systems in the range of hyperarithmetical analysis.  
\end{rem}

\subsubsection{Sequential Heine-Borel theorem}
In this section, we study the sequential version of the Heine-Borel theorem, which does not involve $\CLC$ but does require the following `numerical choice' principle.
\begin{princ}[${\OC}^{0,0}$] 
For any increasing sequence of open sets $(O_{n})_{n\in \N}$ in $\R$: 
\be\label{zata}
(\forall n\in \N)(\exists m\in \N)([-n, n]\subset O_{m})\di (\exists g\in \N^{\N})(\forall n\in \N)([-n, n]\subset O_{g(n)}).
\ee
\end{princ}
By Theorem \ref{henckb}, $\OC^{0,0}$ is not provable from $\CLC$ and much stronger systems.  
We have the following theorem where item \eqref{fa} is `one half' of the Hahn-Kat\v{e}tov-Tong insertion theorem \cite{hahn1, kate,tong}.
\begin{thm}\label{henchmen}~
Over $\RCAo$, the following are equivalent.
\begin{enumerate}
\renewcommand{\theenumi}{\alph{enumi}}
\item $\WKL_{0}$ plus: any usco function $f:\R\di \R$ is bounded above by some continuous $g:\R\di \R$.\label{fa}
\item for a sequence of usco functions $(f_{n})_{n\in \N}$ on $[0,1]$, there is $g\in \N^{\N}$ such that $f_{n}(x)\leq g(n)$ for all $n\in \N, x\in [0,1]$.\label{fap}
\item Let $(O_{n})_{n\in \N}$ be a sequence of open sets such that $\cup_{n\in \N}O_{n}$ covers $\R$.  There is $g\in \N^{\N}$ such that for all $n\in \N$, $\cup_{m\leq g(n)}O_{m}$ covers $[-n, n]$.\label{fb}
\item $(\HBC_{\s}^{\seq})$ Let $(O_{n, m})_{n, m\in \N}$ and $(C_{n})_{n\in \N}$ be sequences of resp.\ open and closed sets in $[0,1]$ such that $\cup_{m\in \N}O_{n,m}$ covers $C_{n}$ for all $n\in \N$.  
There is $g\in \N^{\N}$ such that for all $n\in \N$, $\cup_{m\leq g(n)}O_{m}$ covers $C_{n}$. \label{fc}
\item  The combination of the following:\label{fd}
\begin{itemize}
\item any of the items \eqref{itema}-\eqref{itemdz} from Theorem \ref{hench},
\item the principle $({\OC}^{0,0})$.  
\end{itemize}
\end{enumerate}
\end{thm}
\begin{proof}
The equivalence between items \eqref{fa} and \eqref{fap} is straightforward using translations. 
Now assume item \eqref{fa} and let $(O_{n})_{n\in \N}$ be a sequence of open sets such that $\cup_{n\in \N}O_{n}$ covers $\R$.  
Noting Footnote \ref{floef}, apply $\QFAC^{1,0}$ to: 
\be\label{tata}
(\forall x\in \R)(\exists n\in \N)(x\in O_{n}\wedge x\not \in \cup_{i<n}O_{i})
\ee
and let $f:\R\di \R$ be the associated function.
By definition, $f$ is usco and thus bounded above by a continuous $g:\R\di \R$. 
By the (sequential version of) the boundedness principle for continuous functions, there is $h\in \N^{\N}$ such that $h(m)$ is an upper bound for $g$ on $[-m,m]$ for each $m\in \N$, i.e.\ item \eqref{fb} follows by \eqref{tata}.  

\smallskip

We can (sort of) avoid the aforementioned boundedness principle by making the following case distinction:  in case the representations $h_{O_{n}}$ of $O_{n}$ are continuous functions, the latter have (a sequence of) codes by \cite{dagsamXIV}*{Cor.\ 2.5}, and \cite{simpson2}*{II.7.1 and IV.1.6} yields the required $g\in \N^{\N}$ for item \eqref{fb}; in case some representation $h_{O_{n}}$ of $O_{n}$ is discontinuous, we obtain $(\exists^{2})$ by \cite{kohlenbach2}*{Prop.\ 3.12}, and \cite{kohlenbach2}*{Prop.\ 3.14} provides a supremum functional for continuous functions which yields item~\eqref{fb}.

\smallskip

Next, assume item \eqref{fb} and let $f:\R\di \R$ be usco.  
Item \eqref{fa} trivially follows if $f$ is continuous, i.e.\ we may assume the latter to be discontinuous, yielding $(\exists^{2})$ by \cite{kohlenbach2}*{Prop.\ 3.12}.
Essentially by definition, the set $C_{n}:=\{  x\in \R: f(x)\geq n\}$ is closed.   Then $O_{n}:=[0,1]\setminus C_{n}$ is such that $\cup_{n}O_{n}$ covers $\R$.
Let $g\in \N^{\N}$ be the sequence provided by item (b) and note that $f$ is bounded above on $[-n, n]$ by $g(n)$.  
Item \eqref{fa} now readily follows using $(\exists^{2})$.   

\smallskip

Next, to prove item \eqref{fc} from item \eqref{fa}, note that we may assume $(\exists^{2})$ in the same way as in the proof of $\HBC_{\s}$ in Theorem \ref{hench}.  
Now consider the following function:
\be\label{edo2}
f(x):=
\begin{cases}
m & \textup{if $x\in [n+1,n+2]$, $x-(n+1)\in C_{n}$, and $m$ is}\\
~& \textup{the least natural number such that $x-(n+1)\in O_{n,m}$}\\
0 & \textup{ otherwise}
\end{cases},
\ee 
where the closed sets $C_{n}$ are covered by the open coverings $\cup_{m\in \N}O_{n,m}$.
As for \eqref{edo}, $f$ is usco and consider a continuous $g$ such that $f\leq g$ on $\R$.  
By the (sequential version of) the boundedness principle for continuous functions, there is $h\in \N^{\N}$ such that $h(n)$ is an upper bound for $g$ on $[-n,n]$ for each $n\in \N$, i.e.\ item \eqref{fc} follows as $h$ is as required for the latter.  
Item \eqref{fc} readily implies item \eqref{fb} by translating the sets $C_{n}$ to $[n, n+1]$.  Similarly, item \eqref{fap} is equivalent to \eqref{itema} using translations.  
 
\smallskip

The reversals for item \eqref{fd} are immediate by the previous equivalences.
In particular, one need only apply $\OC^{0,0}$ to the conclusion of the other principle at hand to obtain one of the items \eqref{fa}-\eqref{fc}.  
To derive $\OC^{0,0}$ from the latter, note that \eqref{zata} is a special case of item \eqref{fb}.
\end{proof}
Next, we establish the following properties of $\OC^{0,0}$.  Note that by the first item in Theorem \ref{henckb}, $\Z_{2}^{\omega}+\CLC$ cannot prove $\OC^{0,0}$.  
\begin{thm}\label{henckb}~
\begin{itemize}
\item The system $\Z_{2}^{\omega}+\QFAC^{0,1}$ cannot prove $\OC^{0,0}$.
\item The system $\Z_{2}^{\Omega}$ proves $\OC^{0,0}$.
\end{itemize}
\end{thm}
\begin{proof}
The second item is immediate as $\RCAo$ includes $\QFAC^{0,0}$ while $(\exists^{3})$ makes `$[-n, n]\subset O_{m}$' decidable. 
For the first item, we show that the model $\bf{P}$ of $\Z_{2}^{\omega}+\QFAC^{0,1}$ from \cite{dagsamX} satisfies $\neg \OC^{0,0}$.   
To this end, we first briefly introduce $\bf{P}$ in Definition \ref{modelp} and then prove an essential result about {\bf P} in Lemma \ref{VL}.
The first item then follows via a series of claims (Claims \eqref{c1}-\eqref{c7}).  

\smallskip
\noindent
First of all, the aforementioned model ${\bf P}$ is constructed as follows, assuming $\textsf{V=L}$.
\bdefi[The model $\bf{P}$]\label{modelp}~
\begin{itemize} \item Let ${ \SS}^2_\omega = \langle {\SS}^2_n\rangle_{n \in \N}$ where $\SS_{k}^{2}$ decides $\Pi_{k}^{1}$-formulas in Kleene normal form. 
\item Define ${\bf P}_0 =\N$ and for each finite type $\sigma =( \tau_1 , \ldots , \tau_k \rightarrow  0)$ we define ${\bf P}_\sigma$ be as the set of total maps 
\[
\Phi:{\bf P}_{\tau_1} \times \cdots \times {\bf P}_{\tau_k} \rightarrow \N
\] 
computable in ${\SS}^2_\omega$.   Then ${\bf P}_{\sigma}$ is the set of objects of finite type $\sigma$ in $\bf{P}$.
\item Using Gandy selection \(\cite{longmann}\), one verifies that $\QFAC^{0,1}$ holds in $\bf P$ and that ${\bf P}_{1 \rightarrow 0}$ contains an injection $\phi$ of ${\bf P}_1$ into $\N$.
\end{itemize}
\end{defi}
The final property of $\bf{P}$ is used in \cite{dagsamX} to show that $\Z_{2}^{\omega}+\QFAC^{0,1}$ cannot prove the uncountability of $\R$ formulated as `there is no injection from $2^{\N}$ to $\N$'.

\smallskip

Secondly, the following property of $\bf P$ may be of general interest.
\begin{lemma}\label{VL}
In $\bf P$, there is a well-ordering of $\N^\N$.
\end{lemma}
\begin{proof} 
Since we work under the assumption that \textsf{V = L}, we could have used the well-ordering of \textsf{L} restricted to {\bf P}, but we will need the construction below, based on stage comparison (\cite{longmann}), for computations relative to ${\SS}^2_\omega$.

\smallskip

Recall that the injection $\phi$ from Def.\ \ref{modelp} is such that if $\phi(f) = e$, then $e$ is an index for computing $f\in {\textbf{P}_{1}}$ from ${\SS}^2_\omega$. 
This induces an \emph{ordinal rank} $||f||$ on each $f \in {\bf P}_1$, the rank of this computation. We then define 
\[ 
f \preceq g \leftrightarrow \Big[ ||f|| < ||g|| \vee \big(||f|| = ||g|| \wedge \phi(f) \leq \phi(g)\big)\Big].
\]
Due to the stage comparison property of computations relative to a normal functional of type 2, the previous is all computable in ${\SS}^2_\omega$, and thus the well-ordering $\preceq$ is in the model $\bf P$.
\end{proof}
Thirdly, it is well-known that $\N^\N$ and $[0,1)$ are order-isomorphic, with an arithmetically defined isomorphism.  Moreover, the standard topology on $\N^\N$ corresponds to the topology on $[0,1)$ induced by half-open intervals $[p,q)$ with rational endpoints. In the construction below, we will consider $x$ both as an element of $\N^\N$ and as an element of $[0,1)$, which one will be clear from the context. We will work inside the model $\bf P$ and let $\phi$ be the injection from Def.\ \ref{modelp}.

\smallskip

Let $A$ be the range of $\phi$, i.e.\ $A = \{\phi(x) : x \in \N^\N\}$. Let $h:\N_1 \rightarrow \N_0$ enumerate $A$ in increasing order, and let $y_n = (\phi \circ h)^{-1}(n)$, i.e.\ $h(\phi(x_n)) = n$. 
Now define $g\in \N^{\N}$ as follows: $g(0) := 0$ and for $n\geq 1$ we define
\be\label{defg}\textstyle
g(n) := \sum_{k \leq n}y_k(n) + 1.
\ee
If $k < n$ we have that $g(n) > y_k(n)$, so $g$ will dominate  each element in ${\bf P}_1$ for all but finitely many inputs. Thus, no function dominating $g$ can be in ${\bf P}_1$. Our aim is to construct open sets $O_m$ in such a way that the assumption in $\OC^{0,0}$ is satisfied in $\bf P$, but that any $g'$ satisfying the conclusion will dominate  $g$ from \eqref{defg} in infinitely many points. Hence, $g'$ cannot be in the model $\bf P$.

\smallskip

Now, let $x \in [0,1)$ be given. For $n \geq 1$, we will define the relation $n - 1 + x  \not \in O_m$  computably in ${\SS}^2_\omega$, prove that the assumption in $\OC^{0,0}$ is satisfied, and then observe that the conclusion cannot be satisfied in ${\bf P}_1$. We need the following definitions. 
\begin{itemize}
\item Define $A^x := \{\phi(y) : y \preceq x\}$.
\item Let $h^x$ enumerate $A^x$ in increasing order, again starting with 1.
\item Let $n^x$ be such that $h^x(n^x) = \phi(x)$.
\item Define $y^x_k := (\phi \circ h^x)^{-1}(k)$.
\end{itemize}
We leave $n-1 +x$ out of $O_m$ if $n = n^x$ and $m <  \sum_{k \leq n}y^x_k(n)+1$, otherwise it is in. All negative reals will be in each $O_m$.
\bigskip

Since each $z \geq 0$ can be written, in a unique way, as $n-1 + x$, where $x \in [0,1)$ and $n \geq 1$, the definition of $O_m$ is complete. We will now prove the desired properties of $O_m$ through a sequence of claims, as follows. 
\begin{claim}\label{c1}
Let $x_1 \prec x_2$ be such that $n^{x_1} = n^{x_2} = n$. Then $\phi(x_2) < \phi(x_1)$. 
\end{claim}
\begin{proof} Since $A^{x_1} \subset A^{x_2}$ we must have that $h^{x_2} \leq h^{x_1}$, so, if $\phi(x_2)$ is the $n$-th element in $A^{x_2}$ while $\phi(x_1)$ is the $n$-th element in the smaller $A^{x_1}$, then we must have that $\phi(x_2) \leq \phi(x_1)$. Injectivity of $\phi$ ensures that the order is strict. \end{proof}
\begin{claim}\label{c2} 
For each $n$, there are at most finitely many $x$ with $n = n^x$
\end{claim}
\begin{proof} 
If there are infinitely many such $x$, we obtain an infinite increasing (relative to $\prec$) sequence of such elements, which contradicts Claim \ref{c1}. 
\end{proof}
\begin{claim}\label{c3} 
Each set $O_m$ is open and if $m_1 \leq m_2$, then $O_{m_1} \subseteq O_{m_2}$.
\end{claim}
\begin{proof} 
By Claim \ref{c2}, the complement of $O_m$ over each interval $[n-1,n))$ is finite, since in this interval we only leave out points of the form $n -1 + x$ where $n = n^x$.
Hence $O_{m}$ is open, while the other part follows by definition.
\end{proof}

\begin{claim}\label{c5} For each $n$ there is an $m$ such that $[-n,n] \subseteq O_m$.\end{claim}
\begin{proof} For each $k$ and $x$ such that  $n^x = k+1$, we have an explicit upper bound for when $k-1 + x$ will enter $O_m$. By Claim \ref{c2} there are only finitely many such $k-1 + x \leq n$ , so there must be an $m$ such that they all have entered $O_m$. All points $k-1+x$ where $k \neq n^x$ are in all $O_m$ by construction. \end{proof}

\begin{claim}\label{c5*} 
The following set is infinite:
\[
B=\{n\in A :(\forall x)(\phi(x)>n) \di \phi^{-1}(n) \prec x \}.
\]
\end{claim}
\begin{proof}
Let $\prec^*$ be the well-ordering of $A$ induced by $\prec$ and $\phi$.
\begin{enumerate}
\item $b_1$ is the $\prec^*$-least element of $A$.
\item $b_{k+1}$ be the $\prec^*$-least element of $\{a \in A : b_k < a\}$.
\end{enumerate}
This enumerates $B$ in increasing order, both with respect to $\prec^*$ and $<_{\N}$.  Since $A$ does not have a $\prec^*$-largest element, the enumeration goes on through $\N$.
\end{proof}
\begin{claim}\label{c6} If $\phi(x) = n \in B$, then $n^x = n$ and $y^x_k = y_k$ for all $k \leq n$. \end{claim}
\begin{proof} 
The formula $\phi(x) \in B$ just means that $A^x \cap \{0 , \ldots , \phi(x)\} = A \cap \{0 , \ldots , \phi(x)\}$, so the claim is immediate. 
\end{proof}
\begin{claim}\label{c7} If $n \in B$ and $[n-1,n) \subseteq O_m$ then $g(n) \leq m$. \end{claim}
\begin{proof} Let $n \in B$ and choose (the unique) $x$ such that $\phi(x) = n$. By the construction we leave $n^x - 1 + x$ out of $O_m$ unless $m \geq  \sum_{k \leq n}y^x_k(n)+1$. By Claim \ref{c6} this sum is exactly $g(n)$. \end{proof}

Combining Claims \ref{c3}, \ref{c5*} and \ref{c7} we see that ${\bf P}$ does not satisfy the choice principle $\OC^{0,0}$, i.e.\ the proof of Theorem \ref{henckb} is complete.
\end{proof}
\noindent
We conjecture that $\Z_{2}^{\omega}+\OC^{0,0}$ cannot prove $\CLC$. 

\smallskip
\noindent
Finally, we discuss the coding of usco functions in the light of our results.
\begin{rem}\label{indep}\rm
Semi-continuous functions are studied in \cite{ekelhaft, pekelhaft} using second-order representations.
The latter amount to including a Baire 1 representation, i.e.\ a sequence of continuous functions with pointwise limit the function at hand. 
We argue that this coding is problematic for two reasons, as follows.

\smallskip

Firstly, based on the results in \cite{dagsamXIV}, one readily shows that over $\ACAo$, the third-order statements `open sets as in Def.\ \ref{bchar} have RM-codes' and `usco functions are Baire 1' are equivalent.   
In this light, the coding of usco functions from \cite{ekelhaft, pekelhaft} seems problematic, as the associated coding principle `usco functions are Baire 1' is stronger than the four new `Big' systems studied in \cite{dagsamXI, samBIG, samBIG2}, following \cite{samBIG2}*{Figure 1}.  To be absolutely clear, adopting the coding of usco functions as in \cite{ekelhaft, pekelhaft}, one obfuscates the many new equivalences in third-order arithmetic involving the uncountability of $\R$ (\cite{samBIG}), Jordan's decomposition theorem (\cite{dagsamXI}), the Baire's category theorem (\cite{samBIG2}), and Tao's pigeon hole principle for measure (\cite{samBIG2}). 

\smallskip

A second observation is based on Remark \ref{hecho}.   By the latter, the principle $\CLC$ and the associated principle in Theorem \ref{henchmen2} give rise to 
rather natural systems \emph{in the range of hyperarithmetical analysis}.  The coding from \cites{ekelhaft, pekelhaft} of course destroys this status following \cite{dagsamXIV}*{Theorem 2.9}.  
In other words, certain properties of semi-continuous functions show a natural connection to hyperarithmetical analysis, which is destroyed by the coding in \cites{ekelhaft, pekelhaft}.
\end{rem}

\subsection{Variations}\label{varvka}
We discuss some variations of the above results based on \emph{weak weak K\"onig lemma} (Section \ref{zol}), Cousin's lemma (Section \ref{zol2}), and the Lebesgue number lemma (Section \ref{zol3}).

\subsubsection{Weak weak K\"onig's lemma}\label{zol}
The principle $\WWKL_{0}$ consists of $\RCA_{0}$ plus \emph{weak weak K\"onig's lemma} (see \cite{simpson2}*{X.1.7}), which is the restriction of $\WKL_{0}$ to trees of positive measure.  
The (rather limited) RM of $\WWKL_{0}$ includes a version of the Vitali covering theorem (\cite{simpson2}*{X.1.13}) and some basic theorems from analysis (\cite{sayo}).  
The proof of Theorem \ref{hench} can be adapted to show that over $\RCAo$, the following are equivalent.
\begin{itemize}
\item \(Vitali\) for a sequence $(O_{n})_{n\in \N}$ of open sets such that $\cup_{n\in \N}O_{n}$ covers $[0,1]$ and $k\in \N$, there is $n_{0}\in \N$ such that $\cup_{n\leq n_{0}}O_{n}$ has total length $>1-\frac{1}{2^{k}}$, \label{itemb2}
\item \(weak Cantor intersection theorem\) for a sequence $(C_{n})_{n\in \N}$ of closed sets having \textbf{positive} measure and with $C_{n+1}\subseteq C_{n}\subseteq [0,1]$ for all $n\in \N$, $\cap_{n\in \N}C_{n}\ne \emptyset$,\label{itembp2}
\item for usco $f:[0,1]\di \R$ with supremum \textbf{and} essential supremum both equal to $y$, there is $x\in [0,1]$ with $f(x)=y$.\label{itemd2}
\end{itemize}
We believe there to be more equivalences based on the Riemann integral as in \cite{sayo}. 
The sequential versions of the above items behave in the same way as for $\WKL_{0}$.  

\subsubsection{Cousin's lemma}\label{zol2}
The well-known \emph{Cousin's lemma} (\cite{bartle}) expresses compactness as follows, noting that Cousin in \cite{cousin1}*{p.\ 22} studies the below kind of coverings of closed sets in the Euclidean plane.    
\begin{princ}[Cousin's lemma]\label{CP2}
For closed $C\subset [0,1]$ and $\Psi: [0,1]\di \R^{+}$, there exist $x_{0}, \dots, x_{k}\in C$ with $C\subset \cup_{i\leq k}B(x_{i}, \Psi(x_{i}))$.  
\end{princ}
Even the restriction of Cousin's lemma to $C=[0,1]$ and $\Psi$ having bounded variation is not provable in $\Z_{2}^{\omega}+\QFAC^{0,1}$ (\cite{dagsamIII, dagsamXI, dagsamXIV}).
By contrast, the RM of $\WKL_{0}$ boasts versions of Cousin's lemma restricted to $C=[0,1]$ and $\Psi$ in well-known function classes, including lower (but not upper) semi-continuity (\cite{dagsamXIV}*{\S2.3}). 
Similar to the above proofs, one proves that the higher items imply the lower ones in $\RCAo$ plus extra induction.  
Fragments of the induction axiom are sometimes used in an essential way in second-order RM (see e.g.\ \cite{neeman}).  
\begin{itemize}
\item A usco function on the unit interval is bounded above. 
\item Cousin's lemma as in Principle \ref{CP2} for lsco $\Psi:[0,1]\di \R^{+}$.
\item The Heine-Borel theorem as in $\HBC$.  
\end{itemize}
One readily verifies that the sequential versions of the \emph{contrapositions} of $\HBC$ and Cousin's lemma behave as in Theorem \ref{henchmen2}.  
The sequential version of Cousin's lemma implies the enumeration principle (that any countable set of reals can be enumerated), which is essentially proved in \cite{dagsamX}*{\S3.1.2}.

\subsubsection{The Lebesgue number lemma}\label{zol3}
We have shown in \cite{dagsamXV} that the Lebesgue number lemma has interesting computational properties:  any functional computing the Lebesgue number of countable open coverings, 
is as strong as $\Omega_{C}$, the functional deciding whether closed sets of reals are empty or not.  
This functional is explosive as $\Omega_{C}+\SS^{2}$ computes $\SS_{2}^{2}$ where the latter decides $\Pi_{2}^{1}$-formulas.  
In the below, we show that logical properties of the Lebesgue number lemma are
a lot more tame.  

\smallskip

First of all, we establish the Lebesgue number lemma in a relatively weak system
\begin{thm}[$\ACAo+\QFAC^{0,1}$]\label{gong}
Let $(O_{n})_{n\in \N}$ be a sequence of open sets such that $[0,1]\subset \cup_{n\in \N}O_{n}$.
Then there is $k\in \N$ such that for all $a, b\in [0,1]$ with $|a-b|<\frac{1}{2^{k}}$, there is $n\in \N$ with $(a, b)\subseteq O_{n}$. 
\end{thm}
\begin{proof}
Let $(O_{n})_{n\in \N}$ be a sequence of open sets such that $[0,1]\subset \cup_{n\in \N}O_{n}$.
Suppose there is no Lebesgue number, i.e.\ 
\be\label{frenk}\textstyle
(\forall k\in \N)(\exists a, b\in \Q\cap [0,1])\big[  |a-b|<\frac{1}{2^{k}}\wedge  \underline{(\forall n\in \N)(\exists x\in (a, b) (x\not\in O_{n})}\big].
\ee
Apply $\QFAC^{0,1}$ to the underlined formula in \eqref{frenk} to obtain:
\[\textstyle
(\forall k\in \N)(\exists a, b\in \Q\cap [0,1])(\exists (x_{n})_{n\in \N})\big[  |a-b|<\frac{1}{2^{k}}\wedge  (\forall n\in \N)(x_{n}\in  (a, b)\wedge x_{n}\not\in O_{n})\big].
\]
Apply $\QFAC^{0,1}$ (modulo $(\exists^{2})$ to decide arithmetical formulas) to obtain sequences of rationals $(a_{n})_{n\in \N}$, $(b_{n})_{n\in \N}$ in $[0,1]$ such that
\be\label{ferong}\textstyle
(\forall k\in \N)(\exists (x_{n})_{n\in \N})\big[  |a_{k}-b_{k}|<\frac{1}{2^{k}}\wedge  (\forall n\in \N)(x_{n}\in  (a_{k}, b_{k})\wedge x_{n}\not\in O_{n})\big].
\ee
Define $(y_{n})_{n\in \N}$ as $\frac{a_{n}+b_{n}}{2}$ and use sequential compactness (available due to $(\exists^{2})\di \ACA_{0}$ and \cite{simpson2}*{III.2}) to obtain $g\in \N^{\N}$ such that $(y_{g(n)})_{n\in \N}$ is a convergent sub-sequence, say with limit $z\in [0,1]$.  
By assumption, $B(z, \frac{1}{2^{N_{0}}})\subset O_{n_{0}}$ for some $n_{0}, N_{0}\in \N$.   Hence, for large enough $k\in\N$, we have $(a_{g(k)}, b_{g(k)})\subset O_{n_{0}}$, contradicting \eqref{ferong} and establishing the Lebesgue number lemma.    
\end{proof}
Secondly, an equivalence now readily follows assuming a small fragment of induction, namely the boundedness principle $B\Pi$.  
Fragments of the induction axiom are sometimes used in an essential way in second-order RM (see e.g.\ \cite{neeman}).  
\begin{princ}[$B\Pi$]
For $A(n,m)\equiv(\forall f\in \N^{\N})(Y(f, m,n)=0)$ and $k\in \N$:
\[
(\forall m\leq k)(\exists n\in \N)A(m,n)\di (\exists n_{0}\in \N)(\forall m\leq k)(\exists n \leq n_{0})A(m,n).
\]
\end{princ}
\begin{cor}[$\RCAo+\QFAC^{0,1}+B\Pi$]
The following are equivalent. 
\begin{itemize}
\item A usco function $f:[0,1]\di \R$ is bounded above. 
\item The principle $\WKL_{0}$.
\item The Lebesgue number lemma as in Theorem \ref{gong}.
\end{itemize}
We only need $B\Pi$ for proving the first item from the third item. 
\end{cor}
\begin{proof}
The equivalence involving the first two items is proved in Theorem \ref{hench}. 
To derive the third item in $\RCAo+\WKL+\QFAC^{0,1}$, consider $\neg(\exists^{2})\vee (\exists^{2})$.  Use Theorem \ref{gong} in the latter case, while all functions 
are continuous in the former case by \cite{kohlenbach2}*{Prop.\ 3.12}.  Thus, all open sets have codes by \cite{dagsamXIV}*{Cor.\ 2.5} 
and the second-order proof of the countable Heine-Borel theorem as in \cite{simpson2}*{IV.1} goes through.
Thus, we obtain a finite sub-covering and Lebesgue number exist.  

\smallskip

Finally, to derive item \eqref{itemb} in Theorem \ref{hench} from the third item of the corollary, fix an open covering $(O_{n})_{n\in \N}$ of $[0,1]$ and let $N\in \N$ be such that $\frac{1}{2^{N}}$ is a Lebesgue number.  Hence, we have the following:
\[\textstyle
 (\forall i< 2^{N+1})(\exists n\in \N)\big[(\frac{i}{2^{N+1}},\frac{i+1}{2^{N+1}} )\subset O_{n}\big].
 \]   
The upper bound $n_{0}$ on $n$ provided $B\Pi$ is such that $\cup_{n\leq n_{0}}O_{n}$ covers $[0,1]$.
\end{proof}
We probably can do with less than $B\Pi$, namely a version of $\OC^{0, 0}$ where the outermost universal quantifier $(\forall n\in \N)$ in the antecedent is replaced by $(\forall n\leq k)$ for fixed $k\in \N$.
The sequential version of the Lebesgue number lemma seems provable using $\OC^{0,0}$ and a version of $B\Pi$, but the details are messy.  

\begin{ack}\rm 
We thank Ulrich Kohlenbach for his valuable advice.  
The second author was supported by the \emph{Klaus Tschira Boost Fund} via the grant Projekt KT43.
\end{ack}

\bye